\documentclass[11pt]{amsart}
\usepackage{lmodern}
\usepackage{amsmath, amsthm, amssymb, amsfonts}
\usepackage[normalem]{ulem}
\usepackage{hyperref}
\usepackage{enumitem} 

\usepackage{verbatim} 
\usepackage{longtable}

\usepackage{mathtools}

\usepackage{tikz}
\usetikzlibrary{decorations.pathmorphing}
\tikzset{snake it/.style={decorate, decoration=snake}}

\usepackage{caption}

\usepackage{tikz-cd}
\usetikzlibrary{arrows}

\theoremstyle{plain}
\newtheorem{thm}{Theorem}[section]
\newtheorem{cor}[thm]{Corollary}
\newtheorem{lem}[thm]{Lemma}
\newtheorem{prop}[thm]{Proposition}
\newtheorem{conj}[thm]{Conjecture}

\theoremstyle{definition}
\newtheorem{defn}[thm]{Definition}

\theoremstyle{remark}
\newtheorem{rmk}[thm]{Remark}

\newcommand{\BA}{{\mathbb{A}}}

\newcommand{\BC}{{\mathbb{C}}}

\newcommand{\BH}{{\mathbb{H}}}

\newcommand{\BP}{{\mathbb{P}}}
\newcommand{\BQ}{{\mathbb{Q}}}

\newcommand{\BZ}{{\mathbb{Z}}}

\newcommand{\CP}{{\mathcal P}}

\newcommand{\Fp}{{\mathfrak{p}}}

\DeclareFontFamily{OT1}{rsfs}{}
\DeclareFontShape{OT1}{rsfs}{n}{it}{<-> rsfs10}{}
\DeclareMathAlphabet{\curly}{OT1}{rsfs}{n}{it}

\usepackage{tikz}
\usepackage{lmodern}
\usetikzlibrary{decorations.pathmorphing}

\title[Perverse filtrations for generalized Kummer varieties]{Perverse filtration for generalized Kummer varieties of fibered surfaces}
\author{Zili Zhang}
\address{School of Mathematical Sciences, Tongji University, Shanghai 200092, China}
\email{zhangzili@tongji.edu.cn}
\date{\today}

\begin{document}

\maketitle
\begin{abstract}
    Let $A\to C$ be a proper surjective morphism from a smooth connected quasi-projective commutative group scheme of dimension 2 to a smooth curve. The construction of generalized Kummer varieties gives a proper morphism $A^{[[n]]}\to C^{((n))}$. We show that the perverse filtration associated with this morphism is multiplicative.
\end{abstract}
\tableofcontents

\section{Introduction}
\subsection{Perverse filtrations}
Let $f:X\to Y$ be a proper morphism between smooth quasi-projective varieties. Let $^\Fp\tau_{\le k}$  be the perverse truncations of the bounded derived category of constructible sheaves $D^b_c(Y)$. Applying $^\Fp\tau_{\le k}$ to the pushforward $Rf_*\BQ_X$ gives an increasing filtration in $H^*(X,\BQ)$
\begin{equation}\label{-1}
P_0H^*(X,\BQ)\subset P_1H^*(X,\BQ)\subset\cdots\subset P_{k}H^*(X,\BQ)\subset \cdots\subset H^*(X,\BQ). 
\end{equation}
The filtration (\ref{-1}) is called \emph{the perverse filtration associated with the morphism $f$}. A perverse filtration is called \emph{multiplicative} if the cup product
\[
P_k H^d(X,\BQ)\cup P_l H^e(X,\BQ)\subset P_{k+l}H^{d+e}(X,\BQ), ~~~~k,l,d,e\ge 0.
\]
We say that a perverse filtration \emph{admits a strongly multiplicative splitting} if there is a direct sum decomposition 
\[
H^*(X)=\bigoplus_i G_iH^*(X,\BQ)
\]
splitting the perverse filtration in the sense that
\[
P_kH^*(X,\BQ)=\bigoplus_{i=0}^k G_i H^*(X,\BQ),
\]
such that the cup product
\[
G_iH^d(X,\BQ)\cup G_j H^e(X,\BQ)\subset G_{i+j}H^{d+e}(X,\BQ), ~~~~i,j,d,e\ge 0.
\]

\subsection{The P=W conjecture}
Let $C$ be a smooth projective curve. There are two moduli spaces associated with $C$ and the structure group $\textup{GL}(n,\BC)$. They are the Betti moduli $M_B$ which parametrizes all (twisted) $\textup{GL}(n,\BC)$-representations of the fundamental group $\pi_1(C)$ and the Dolbeault moduli $M_D$ which parametrizes all semi-stable Higgs bundles of degree $d$ and rank $n$. The Dolbeault moduli $M_D$ admits a natural proper morphism, called Hitchin map, to an affine space $\pi:M_D\to\BA$. Simpson proves that there is a \emph{canonical} diffeomorphism between $M_B$ and $M_D$, and hence there is a canonical isomorphism of cohomology groups 
\begin{equation}\label{-2}
    H^*(M_D,\BQ)= H^*(M_B,\BQ).
\end{equation}

In \cite{dCHM}, de Cataldo, Hausel, and Migliorini conjectured that under the identification (\ref{-2}), the perverse filtration associated with $\pi$ coincides with the mixed Hodge-theoretic weight filtration, \emph{i.e.}
\begin{conj}\cite{dCHM}\label{-3}
\[
P_k H^*(M_D,\BQ)=W_{2k} H^*(M_B,\BQ)=W_{2k+1} H^*(M_B,\BQ), ~~~~k\ge 0.
\]
\end{conj}

Conjecture \ref{-3} is referred to as the $P=W$ conjecture. It is proved by Maulik-Shen \cite{MS} and Hausel-Mellit-Minets-Schiffmann \cite{HMMS} independently. The properties of the perverse filtration and weight filtration are quite different. Since the Hodge-theoretic weight filtration for arbitrary varieties is always multiplicative, the $P=W$ conjecture suggests that the perverse filtration associated with Hitchin maps are multiplicative. In fact, by studying the ring generators of $H^*(M_B,\BQ)= H^*(M_D,\BQ)$ described by Markman \cite{Markman}, de Cataldo, Maulik, and Shen calculate the perversity of the ring generators, and show that the $P=W$ conjecture is equivalent to the multiplicativity of the perverse filtration associated with the Hitchin map $\pi$; see \cite[Proposition 1.6]{dCMS}. However, the multiplicativity of perverse filtration does not hold for general morphisms; see \cite[Example 1.5]{Z}. 

Since the Dolbeault moduli spaces are hyperK\"ahler manifolds and the Hitchin maps are Lagrangian fibrations, the $P=W$ conjecture suggests that the multiplicativity of perverse filtrations holds for Lagrangian fibrations. Shen and Yin verified multiplicativity for compact hyperK\"ahler varieties and Lagrangian fibrations in \cite{SY}. For non-compact cases,
the two proofs \cite{MS} and \cite{HMMS} of the $P=W$ prove in particular the multiplicativity of $\textup{GL}(n)$-Hitchin fibrations. The multiplicativity of perverse filtrations for Hilbert schemes of fibered surfaces is studied in \cite{SZ,Z,Z1}.  In this article, we give a concrete description of the perverse filtrations associated with generalized Kummer varieties of quasi-projective fibered surfaces, and show that the perverse filtration is multiplicative and admits a natural strongly multiplicative splitting. One interesting aspect of this result is that the cohomology ring of generalized Kummer variety is \emph{not} generated by tautological classes. Thus this result may hint at the multiplicativity for other non-tautologically generated spaces, such as the moduli of $\textup{SL}(n,\BC)$-Higgs bundles.

\subsection{Generalized Kummer varieties for fibered surfaces}
Let $A$ be a smooth connected (not necessarily compact) commutative group scheme of dimension 2. Denote $A^{[n]}$ the Hilbert scheme of $n$ points on $A$. Let $A^{(n)}$ be the $n$-fold symmetric product of $A$. There is natural summation map $+:A^{(n)}\to A$. The kernel of the composition of the Hilbert-Chow morphism and the summation
\[
A^{[n]}\to A^{(n)}\to A
\]
is called the generalized Kummer variety, denoted as $A^{[[n]]}$. The generalized Kummer variety\footnote{Throughout this paper, we use the term ``generalized Kummer variety $A^{[[n]]}$'' without assuming $A$ to be compact. } $A^{[[n]]}$ is a smooth variety of dimension $2n-2$. When $A$ admits a proper surjective morphism to a curve $C$, there is a natural proper fibration constructed as follows. Consider the composition
\begin{equation}\label{-4}
A^{[[n]]}\hookrightarrow A^{[n]}\to A^{(n)} \to C^{(n)}.
\end{equation}
Let $C^{((n))}$ be the image of (\ref{-4}). Then $C^{((n))}$ is a variety of dimension $n-1$ (Proposition \ref{dim}), and the morphism $A^{[[n]]}\to C^{((n))}$ is a proper morphism. Our main result is

\begin{thm}[Theorem \ref{main}]
Let $f:A\to C$ be a proper surjective morphism from a connected quasi-projective
commutative group scheme $A$ of dimension 2 to a quasi-projective curve $C$. Then the perverse filtration associated with the induced morphism $h':A^{[[n]]}\to C^{((n))}$ is multiplicative.
\end{thm}

There are three main ingredients: (1) the classification of fibered group surfaces (Proposition \ref{str}), (2) the cup product formula (Theorem \ref{cup}), and (3) the description of the perverse filtration associated with $A^{[[n]]}\to C^{((n))}$ (Proposition \ref{aaa}). 

This article is organized as follows. In Section 2, we recall the definition and properties of perverse filtrations on the cohomology groups and compactly supported cohomology associated with general proper morphisms between smooth varieties. In Section 3, we recall the properties of Hilbert schemes of points of fibered surfaces. In Section 4, we first classify quasi-projective smooth commutative group schemes of dimension 2 properly fibered  over smooth curves. Then we calculate the perverse filtrations of the generalized Kummer varieties of fibered surfaces, and prove that they are multiplicative.

\subsection{Acknowledgements} I am grateful to Michel Brion for help on non-compact 2D algebraic groups, and for pointing out references. I thank Lie Fu and Shizhang Li for helpful discussions. I also thank the anonymous referees' careful work and helpful suggestions. I am partially supported by the Fundamental Research Funds for the Central Universities.

\section{Perverse filtrations}
Let $f:X\to Y$ be a proper morphism between smooth quasi-projective varieties. The perverse $t$-structure on $D^b_{c}(Y)$, the bounded derived category of constructible sheaves on $Y$, gives a morphism
\[
^\Fp\tau_{\le k}\left(Rf_*\BQ_X[\dim X-r(f)]\right)\to Rf_*\BQ_X[\dim X-r(f)],
\]
where 
\[
r(f)=\dim X\times_Y X-\dim X
\]
is the defect of semismallness of $f$. Applying hypercohomology yields a morphism in cohomology groups
\begin{equation} \label{01}
\begin{split}
    \BH^{d-\dim X+r(f)}\left({^\Fp\tau_{\le k}}(Rf_*\BQ_X[\dim X-r(f)])\right)&\to \\
    \BH^{d-\dim X+r(f)}(Rf_*\BQ_X[\dim X-&r(f)])=H^d(X,\BQ).
\end{split}
\end{equation}
We define the subspace $P_k^fH^d(X,\BQ)$ to be the image of (\ref{01})\footnote{The shift $[\dim X-r(f)]$ is to ensure that the perverse filtration starts at 0.}. We will omit the reference to the morphism $f$ when no confusion arises. Then we have an increasing filtration
\[
 P_0H^d(X,\BQ)\subset P_1H^d(X,\BQ)\subset\cdots \subset H^d(X,\BQ), ~~~~d\ge 0
\]
called the perverse filtration associated with morphism $f$. It is concentrated at the interval $[0,2r]$, \emph{i.e.} $P_{-1}H^d(X,\BQ)=0$ and $P_{2r}H^d(X,\BQ)=H^d(X,\BQ)$ for all $d$. For a nonzero class $\alpha\in H^*(X,\BQ)$, we denote $\Fp^f(\alpha)=k$ if $\alpha\in P_kH^*(X,\BQ)$ and $\alpha\not\in P_{k-1}H^*(X,\BQ)$. We say that the perverse filtration is multiplicative if
\[
P_kH^d(X,\BQ)\cup P_lH^e(X,\BQ)\subset P_{k+l}H^{d+e}(X,\BQ), ~~~~k,l,d,e\ge 0.
\]

Suppose the perverse filtration $P_\bullet H^*(X,\BQ)$ admits a splitting
\[
P_kH^*(X,\BQ)=\bigoplus_{i=0}^k G_iH^*(X,\BQ)
\]
satisfying 
\begin{equation}\label{001}
G_iH^d(X,\BQ)\cup G_jH^e(X,\BQ)\subset G_{i+j}H^{d+e}(X,\BQ), ~~~~i,j,d,e\ge0,
\end{equation}
then we say that the perverse filtration $P_\bullet H^*(X,\BQ)$ admits a \emph{strongly multiplicative splitting}, or the morphism $f$ admits a strongly multiplicative perverse decomposition. It follows from the definition that if $f$ admits a strongly multiplicative perverse decomposition, then the perverse filtration associated with $f$ is multiplicative, but not vice versa. A class $\beta\in G_iH^d(X,\BQ)$ for some $i$ and $d$ is called \emph{pure of perversity $i$ with respect to the splitting $G$}.

We say that a basis $\{\beta_i\}$ of $H^*(X,\BQ)$ is adapted to the perverse filtration $P_\bullet H^*(X,\BQ)$ if 
\[
P_kH^*(X,\BQ)=\left\langle\beta_i\mid\beta_i\in P_kH^*(X,\BQ)\right\rangle.
\]

Similarly, the perverse filtration can also be defined on the compactly supported cohomology by taking $R\Gamma_c$ in (\ref{01}). Then the natural transformation of functors $R\Gamma_c\to R\Gamma$ implies that the forgetful map $\iota:H^*_c(X,\BQ)\to H^*(X,\BQ)$ preserves the perverse filtration, \emph{i.e.} $\iota P_kH^*_c(X,\BQ)\subset P_kH^*(X,\BQ)$.

\begin{prop}\label{dualbasis}
Let $f:X\to Y$ be a proper morphism between smooth quasi-projective varieties. Then there exists a basis $\{\beta_i\}$ of $H^*(X,\BQ)$ adapted to $P_\bullet H^*(X,\BQ)$ and a basis $\{\beta^i\}$ adapted to $P_\bullet H^*_c(X,\BQ)$, such that
\begin{enumerate}
    \item $\{\beta_i\}$ and $\{\beta^i\}$ are dual with respect to the Poincar\'e pairing, i.e. 
    \[
    \langle\beta_i,\beta^j\rangle_X=\delta_{ij}.
    \]
    \item $\Fp(\beta_i)+\Fp(\beta^i)=2r(f)$.
\end{enumerate}
In particular, if $\Fp(\alpha)+\Fp(\beta)< 2r(f)$ for some $\alpha\in H^*(X,\BQ)$ and $\beta\in H^*_c(X,\BQ)$, then $\langle\alpha,\beta\rangle_X=0$.
\end{prop}

\begin{proof}
Fix a (non-canonical) decomposition 
\[
Rf_*\BQ[\dim X-r(f)]\cong \CP_0\bigoplus\cdots\bigoplus \CP_{2r(f)}[-2r(f)],
\]
where $\CP_i$ are perverse sheaves. Then \cite[Proposition 4.1]{Z1} and \cite[Remark 2.9]{Z1} produce the desired bases $\{\beta_i\}$ and $\{\beta^i\}$.
\end{proof}

The relation of push-forward and pull-back of perverse filtrations is described as follows.
\begin{prop} \label{pushpull}
Let $f:X\to A$ and $g:Y\to B$ be two proper morphism between smooth quasi-projective varieties. Let $h:X\to Y$ be a proper morphism. Then for any integer $m$, the following are equivalent.
\begin{enumerate}
    \item $\Fp^f(h^*\beta)\le \Fp^g(\beta)+m$ for all $\beta\in H^*(Y,\BQ)$.
    \item $\Fp^g(h_*\alpha)\le \Fp^f(\alpha)+m+2r(g)-2r(f)$ for all $\alpha\in H^*_c(X,\BQ)$.
\end{enumerate}
\end{prop}

\begin{proof}
By adjunction, for $\alpha\in H^*_c(X,\BQ)$, $\beta\in H^*(Y,\BQ)$, we have
\[
\langle h_*\alpha,\beta\rangle_Y=\langle \alpha,h^*\beta\rangle_X,
\]
where $\langle-,-\rangle$ denotes the Poincar\'e paring.\\
We prove $(1)\Rightarrow(2)$. Suppose $\alpha\in P_kH^*_c(X,\BQ)$. Then for any $\beta\in P_{2r(f)-k-m-1}H^*(Y,\BQ)$, $\Fp(h^*\beta)\le 2r(f)-k-1$. By Proposition \ref{dualbasis}, $\langle \alpha,h^*\beta\rangle_X=0$, and hence $\langle h_*\alpha,\beta\rangle_Y=0$. Let $\{\beta^i\}$ be the basis obtained in Proposition \ref{dualbasis} for the morphism $g:Y\to B$. Suppose $\Fp^g(f_*\alpha)>k+m+2r(g)-2r(f)$, then at least one $\beta^i$ with $\Fp^g(\beta^i)>k+m+2r(g)-2r(f)$ would have nonzero coefficient in the expansion of $f_*\alpha$. Then by Proposition \ref{dualbasis}, $\Fp^g(\beta_i)<2r(f)-k-m$ and $\beta_i$ would have nonzero paring with $f_*\alpha$, contradiction. The proof of the reverse direction is similar.
\end{proof}

\begin{prop}\cite[Proposition 2.1]{Z}\label{Kunneth}
Let $f_1:X_1\to Y_1$ and $f_2:X_2\to Y_2$ be proper morphisms between smooth quasi-projective varieties. Then 
\[
\Fp^{f_1\times f_2}(\alpha_1\boxtimes \alpha_2)=\Fp^{f_1}(\alpha_1)+\Fp^{f_2}(\alpha_2).
\]
\end{prop}

Let $\Gamma:X\to X\times Y$ be the graph of $f:X\to Y$. Since $\Gamma$ is a proper morphism, the Gysin push-forward can be defined via Borel-Moore homology
\[
\Gamma_*: H^*(X,\BQ)\cong H^{BM}_*(X,\BQ)\to H^{BM}_*(X\times Y,\BQ)\cong H^*(X\times Y,\BQ).
\]
Equivalently,  
\begin{equation} \label{02}
\begin{split}
    \Gamma_*:H^*(X,\BQ)\xrightarrow{a}& H^*(X,\BQ)\otimes H^*_c(X,\BQ)\\
    \xrightarrow{\textup{id}\otimes f_*}& H^*(X,\BQ)\otimes H^*_c(Y,\BQ)\xrightarrow{\textup{id}\otimes \iota} H^*(X,\BQ)\otimes H^*(Y,\BQ),
\end{split}
\end{equation}
where $a$ is the adjoint map of the cup product
\[
\cup:H^*(X,\BQ)\times H^*(X,\BQ) \to H^*(X,\BQ)
\]
and $\iota$ is the forgetful map
\[
\iota:H^*_c(Y,\BQ)\to H^*(Y,\BQ).
\]
 
We have the following perversity estimation of the graph of a proper morphism.
\begin{prop} \label{graph}
Let $f:X\to A$ be a proper morphism between smooth quasi-projective varieties. Let $g:Y\to B$ be a proper morphism between smooth connected quasi-projective varieties. Let $h:X\to Y$ be a proper morphism satisfying the following two properties.
\begin{enumerate}
    \item The pull-back $h^*:H^*(Y,\BQ)\to H^*(X,\BQ)$ preserves the perverse filtrations, i.e. $\Fp^f(h^*(\alpha))\le \Fp^g(\alpha)$ holds for all $\alpha$.
    \item The perverse filtration associated with the morphism $f:X\to A$ is multiplicative.
\end{enumerate}  
Let $\Gamma:X\to X\times Y$ be the graph of $h$. Then the push-forward along the closed embedding $\Gamma$ increases the perversity by at most $2r(g)$, i.e. for any class $\gamma\in H^*(X,\BQ)$,
\[
\Fp^{f\times g}(\Gamma_*(\gamma))\le \Fp^{f}(\gamma)+2r(g).
\]
\end{prop}

\begin{proof}
Let $\{\alpha_i\}$ and $\{\alpha^i\}$ be dual bases of $H^*(X,\BQ)$ and $H^*_c(X,\BQ)$ adapted to the perverse filtration associated with $f$ obtained in Proposition \ref{dualbasis}. By (\ref{02}), we have
\[
\Gamma_*(\gamma)=\sum_{i}\gamma\alpha_i\otimes\iota f_*\alpha^i.
\]
Then by Proposition \ref{pushpull}, condition (1) implies
\[
\Fp^g(h_*\alpha^i))\le\Fp^f(\alpha^i)+2r(g)-2r(f).
\]
Now by Proposition \ref{Kunneth}, condition (2) and Proposition \ref{dualbasis},
\[
\begin{split}
\Fp^{f\times g}(\gamma\alpha_i\otimes \iota f_*\alpha^i)\le&\Fp^f(\gamma)+\Fp^f(\alpha_i)+\Fp^f(\alpha^i)+2r(g)-2r(f)\\
\le&\Fp^f(\gamma)+2r(g).
\end{split}
\]
So $\Fp^{f\times g}(\Gamma_*(\gamma))\le \Fp^{f}(\gamma)+2r(g)$.
\end{proof}

As an application, we have the following generalization of \cite[Proposition 3.8]{Z} and \cite[Proposition 5.5]{Z}.

\begin{prop}\label{diag}
Let $f:X\to Y$ be a proper morphism between smooth quasi-projective varieties. Let $f^n:X^n\to Y^n$ be the induced morphism of Cartesian products. Let $\Delta_n:X\to X^n$ be the diagonal embedding. Suppose that the perverse filtration associated with $f$ is multiplicative.  Then
\[
\Fp^{f^n}(\Delta_{n*}(\alpha))\le \Fp^f(\alpha)+2(n-1)r(f).
\]
\end{prop}

\begin{proof}
The $n=2$ case follows from taking $f=g$ and $h=\textup{id}$ in Proposition \ref{graph}. The general $n$ follows from an induction argument using
\[
\Delta_n:X\xrightarrow{\Delta_{n-1}} X^{n-1}\xrightarrow{\Delta_2\times \textup{id}^{n-2}} X^n.
\]
\end{proof}

For later use, we also include the following easy fact.
\begin{prop}\label{finite}
Let $X$, $Y$ and $Z$ be quasi-projective varieties. Let $f:X\to Y$ be a proper morphism and $g:Y\to Z$ be a finite surjective morphism. Then the perverse filtration associated with $f$ is identical with the one associated with $g\circ f$, \emph{i.e.}
\[
P^f_kH^*(X,\BQ)=P^{g\circ f}_kH^*(X,\BQ).
\]
\end{prop}

\begin{proof}
Since an finite morphism is $t$-exact, $g_*$ commutes perverse truncations $^\Fp\tau_{\le k}$ and we have the following commutative diagram
\begin{equation}\label{10001}
    \begin{tikzcd}
^\Fp\tau_{\le k}R(g\circ f)_*\BQ_{X}\arrow[r,"\cong"]\arrow[d]&g_*{^\Fp\tau_{\le k}}Rf_*\BQ_{X}\arrow[d]\\
R(g\circ f)_*\BQ_{X}\arrow[r,equal]&g_*Rf_*\BQ_{X}.
\end{tikzcd}
\end{equation}
By the definition of perverse filtrations (\ref{01}), the diagram ($\ref{10001}$) implies
\[
P^f_{k-\dim X+r(g\circ f)}H^*(X,\BQ)=P^{g\circ f}_{k-\dim X+r(f)}H^*(X,\BQ).
\]
It remains to check $r(f)=r(g\circ f)$, which follows from the standard argument of stratification by fiber dimensions. 
\end{proof}

\section{Hilbert schemes of fibered surfaces}
\subsection{Partitions}
We recall some notations of partitions in this section. We say $\nu=(\nu_1,\nu_2,\cdots,\nu_l)$ is a partition of an integer $n$ if $\nu_1\ge\cdots\ge\nu_l>0$ and $n=\nu_1+\cdots+\nu_l$. The integer $l$ is called the length of the partition $\nu$, also denoted as $l(\nu)$. The greatest common divisor $\gcd(\nu)$ is defined as $\gcd(\nu_1,\cdots,\nu_l)$.
For a given partition $\nu$, denote by $a_i$ the number of times that $i$ appears in the partition. Then $n=a_1+2a_2+\cdots+na_n$ and $l=a_1+\cdots+a_n$. We also write $\nu=1^{a_1}\cdots n^{a_n}$.

\subsection{Products and symmetric products}
Let $X$ be a smooth quasi-projective variety. We denote by $X^n$ the $n$-fold Cartesian product and by $X^{(n)}=X^n/\mathfrak{S}_n$ the symmetric product, where $\mathfrak{S}_n$ is the symmetric group. The elements in $X^{(n)}$ are denoted as $x_1+\cdots+x_n$, where $x_i\in X$. Let $\nu=1^{a_1}\cdots n^{a_n}$ be a partition of $n$. Denote by $X^{\nu}$ or $X^{l(\nu)}$ the Cartesian product $X^{a_1}\times\cdots\times X^{a_n}$. Denote by $X^{(\nu)}$ the product of symmetric products $X^{(a_1)}\times \cdots\times X^{(a_n)}$. Let $\mathfrak{S}_\nu=\mathfrak{S}_{a_1}\times\cdots\times \mathfrak{S}_{a_n}$. Then $X^{(\nu)}=X^{\nu}/\mathfrak{S}_{\nu}$. By a theorem of Grothendieck, 
\begin{equation}\label{989}
H^*(X^{(n)},\BQ)=H^*(X^n,\BQ)^{\mathfrak{S}_n}=\textup{Sym}^n H^*(X,\BQ).
\end{equation}
We also have a closed embedding 
\begin{eqnarray*}
   \iota^{(\nu)}:&X^{(\nu)}&\to X^{(n)}\\
    &\displaystyle\left(\sum_{j=1}^{a_1}x_{1j},\cdots,\sum_{j=1}^{a_n}x_{nj}\right)&\mapsto \sum_{i=1}^n\sum_{j=1}^{a_i}ix_{ij}.
\end{eqnarray*}

Let $f:X\to Y$ be a proper morphism between smooth quasi-projective varieties. Denote by $f^n:X^n\to Y^n$ and $f^{(n)}:X^{(n)}\to Y^{(n)}$ the induced maps. By Proposition \ref{Kunneth}, The perverse filtrations associated with $f^n$ is
\begin{equation} \label{987}
    P_k^{f^n}H^*(X^n,\BQ)=\left\langle\alpha_1\boxtimes\cdots\boxtimes\alpha_n\mid \sum_{i=1}^n\Fp(\alpha_i)\le k.\right\rangle
\end{equation}
Taking the $\mathfrak{S}_n$-invariant part of (\ref{987}) yields the filtration associated with $f^{(n)}$. 

Let $\nu$ be a partition, then we also have the induced maps $f^{\nu}:X^\nu\to Y^\nu$ and $f^{(\nu)}:X^{(\nu)}\to Y^{(\nu)}$. The perverse filtration associated with $f^{\nu}$ and $f^{(\nu)}$ are computed easily by Proposition \ref{Kunneth} and (\ref{989}).

\subsection {Perverse filtration for Hilbert schemes}
Let $f:S\to C$ be a proper surjective morphism from a smooth quasi-projective surface $S$ to a smooth curve $C$. Following the notations in Section 3.1 and 3.2, we have the diagram
\[
\begin{tikzcd}
 & & S^{[n]}\arrow[d,swap,"\pi"]\arrow[dd,bend left,"h"]\\
S^{\nu}\arrow[r,"/\mathfrak{S}_\nu"]\arrow[d,swap,"f^{\nu}"]& S^{(\nu)}\arrow[r,"\iota^{(\nu)}_S"]\arrow[d,swap,"f^{(\nu)}"]& S^{(n)}\arrow[d,swap,"f^{(n)}"]\\
C^{\nu}\arrow[r,"/\mathfrak{S}_\nu"]& C^{(\nu)}\arrow[r,"\iota^{(\nu)}_C"]& C^{(n)}.
\end{tikzcd}
\]
\begin{thm}\cite[Theorem 2, Theorem 4]{GS} \label{hilb}
Let $S$ be a smooth quasi-projective surface. Then there is a canonical decomposition
\begin{equation} \label{0123}
    R\pi_*\BQ_{S^{[n]}}[2n]=\bigoplus_{\nu\vdash n}\iota^{(\nu)}_{S*}\BQ_{S^{(\nu)}}[2l(\nu)].
\end{equation}

In particular, there is a decomposition on the cohomology.
\begin{equation} \label{0124}
    H^*(S^{[n]},\BQ)=\bigoplus_{\nu\vdash n} H^*(S^{(\nu)},\BQ)[2l(\nu)-2n].
\end{equation}
\end{thm}

The perverse filtration associated with $h:S^{[n]}\to C^{(n)}$ is calculated in \cite{Z}.

\begin{prop}\cite[Corollary 4.14]{Z} \label{4.7}
Under the canonical isomorphism (\ref{0124}), the perverse filtration
\begin{equation} \label{0125}
    P_k^h H^*(S^{[n]},\BQ)= \bigoplus_{\nu\vdash n}P_{k+l(\nu)-n}^{f^{(\nu)}} H^*(S^{(\nu)},\BQ)[2l(\nu)-2n].
\end{equation}
\end{prop}

\begin{cor} \label{decomp}
Let $\alpha_\nu\in H^*(A^{(\nu)},\BQ)$ be a cohomology class. Denote by $\alpha_\nu^{[n]}\in H^*(A^{[n]},\BQ)$ its image in $H^*(A^{[n]},\BQ)$ via the decomposition (\ref{0124}). Then 
\[
\Fp^h(\alpha_\nu^{[n]})=\Fp^{f^{(\nu)}}(\alpha_\nu)+n-l(\nu).
\]
Let $\alpha=\sum_\nu \alpha_\nu^{[n]}$. Then 
\[
\Fp^h(\alpha)=\max_\nu\left\{\Fp^h\left(\alpha_\nu^{[n]}\right)\right\}=\max_\nu\left\{\Fp^{f^{(\nu)}}(\alpha_\nu)+n-l(\nu)\right\}.
\]
\end{cor}

The following result is a slight generalization of \cite[Theorem 4.18]{Z} and \cite[Theorem 5.6]{Z}.
\begin{thm} \label{multi}
Let $f:S\to C$ be proper surjective morphism from a smooth quasi-projective surface $S$ with trivial canonical bundle to a smooth curve $C$. Let $h:S^{[n]}\to C^{(n)}$ be the induced morphism. Suppose further that $S$ admits a smooth compactification $\bar{S}$ such that the restriction $H^*(\bar{S})\to H^*(S)$ is surjective. Then the perverse filtration associated with $h$ is multiplicative.
\end{thm}

\begin{proof}
The projective case and the five families of Hitchin system (\cite[Section 5.1]{Z}) case are treated in \cite[Theorem 4.18]{Z} and \cite[Theorem 5.6]{Z}, respectively. Their proofs are similar, which are based on the following three ingredients. 
\begin{enumerate}
    \item The cup product formula for Hilbert schemes $S^{[n]}$ with trivial canonical bundle. \cite[Theorem 3.2]{LS} for the projective case and \cite[Proposition 4.10]{Z} for the quasi-projective case.
    \item The description of the perverse filtration Proposition \ref{4.7}.
    \item The diagonal estimation \cite[Proposition 3.8]{Z} for projective case and \cite[Proposition 5.5]{Z} for the five families of Hitchin systems.
\end{enumerate}
In our generality, the cup product formula and the description of the perverse filtration still hold. The diagonal estimation is Proposition \ref{diag}.   
\end{proof}

\section{Generalized Kummer varieties for fibered surfaces}

\subsection{Fibered commutative group schemes of dim 2}
We say a morphism $f:A\to C$ of quasi-projective varieties satisfying the condition ($\dagger$) if the following holds.

($\dagger$) $f:A\to C$ is a proper surjective morphism from a connected smooth quasi-projective commutative group scheme $A$ of dimension 2 to a quasi-projective curve $C$.

The following proposition classifies the morphisms satisfying ($\dagger$).

\begin{prop} \label{str}
Let $f:A\to C$ be a morphism satisfying ($\dagger$). Then its Stein factorization  $A\xrightarrow{f'} B\xrightarrow{g} C$ belongs to one of the following cases.
\begin{enumerate}
    \item The surface $A$ is an abelian surface, $B$ is an elliptic curve, and $C$ is $\BP^1$ or an elliptic curve. The morphism $f'$ is a group homomorphism and $g$ is a finite morphism.
    \item The surface  $A=E\times \BC$, and $B=C=\BC$ where $E$ is an elliptic curve. The morphism $f'$ is the natural projection and $g$ is a finite morphism. 
    \item The surface $A=(E\times \BC^*)/\Gamma$ and $B=C=\BC^*$, where $E$ is an elliptic curve and $\Gamma$ is a finite cyclic group. The morphism $f'$ is an equivariant $\Gamma$-quotient of the natural projection $E\times \BC^*\to\BC^*$ and $g$ is a finite morphism.
    \end{enumerate}

\end{prop}

\begin{proof}
Since $f:A\to C$ is surjective and $g:B\to C$ is finite,  $B$ is a curve. 
By \cite[Tag 03H0]{Stack}, the curve $B$ is the relative normalization of $C$ in $A$, and hence is normal \cite[Tag 0BAK]{Stack}. So $B$ is a smooth curve.

When $B$ is compact, $A$ is an abelian variety. Then the genus $g(B)\le 1$. Since the generic fiber of any fibration $A\to\BP^1$ is disconnected (to be proved in Lemma \ref{fu}), $B$ must be an elliptic curve. Then $f':A\to B$ is a group homomorphism after a choice of the origin in $B$, \cite[Proposition V.12]{Beau}.

When $B$ is non-compact, it is an affine curve and $f':A\to B$ is the affinization of $A$ \cite[Section 3.2]{B} and $f'$ is a group homomorphism. So $B=\BC$ or $\BC^*$, and $A$ is an extension of $B$ with an elliptic curve $E$,
\begin{equation}\label{exact}
1\to E\to A\to B\to 1.    
\end{equation}
When $B=\BC$, by Chevalley's theorem \cite[Theorem 2]{B}, there exists a unique exact sequence of algebraic groups
\[
1\to X\xrightarrow{s} A\to Y\to 1,
\]
where $X$ is a smooth affine group scheme and $Y$ is proper. So $X=\BC$ and $Y$ is an elliptic curve. Since any non-trivial additive group endomorphism of $\BC$ (as $\BC$-varieties) is an isomorphism, the composition 
\[
\BC\xrightarrow{s} A\xrightarrow{f'}\BC
\]
is an isomorphism. Therefore the exact sequence (\ref{exact}) splits and $A=E\times\BC$.

When $B=\BC^*$, a similar argument gives a composition $\BC^*\to A\to \BC^*$. Since any non-trivial endomorphism of $\BC^*$ is given by $z\to z^n$ for some nonzero integer $n$. We may assume $n>0$; precomposing the isomorphism $z\to z^{-1}$ otherwise. The exact sequence (\ref{exact}) splits after a base change $z\to z^n$. Therefore, $f:A\to \BC^*$ is an equivariant quotient of the projection $E\times\BC^*\to \BC^*$ by a finite cyclic group $\BZ/n\BZ$. 
\end{proof}

\begin{lem}\label{fu}
Let $A$ be an abelian surface. Then any surjective morphism $f:A\to \BP^1$ has disconnected generic fiber. 
\end{lem}

\begin{proof}
By \cite[Lemma 1.1]{Barth}, any irreducible component of any fiber is an elliptic curve,  and any such two elliptic curves  are algebraically equivalent. Since the arithmetic genus of the fibers are constant, the number of elliptic curve components on fibers are constant. If the generic fiber were connected, then the generic smoothness would imply that $f:A\to\BP^1$ would be an smooth morphism with all fibers being elliptic curves. Then the decomposition theorem 
\[
\begin{split}
Rf_*\BQ_{A}=&R^0f_*\BQ_A\bigoplus R^1f_*\BQ_A[-1]\bigoplus R^2f_*\BQ_A[-2]\\
=&\BQ_{\BP^1}\bigoplus \BQ_{\BP^1}^{\oplus 2} [-1]\bigoplus \BQ_{\BP^1}[-2]
\end{split}
\]
would imply $H^*(A,\BQ)=H^*(\BP^1,\BQ)\oplus H^*(\BP^1,\BQ)^{\oplus 2}[-1]\oplus H^*(\BP^1,\BQ)[-2]$, which is a contradiction.
\end{proof}

\begin{prop} \label{4.3}
Let $f:A\to C$ be a morphism satisfying $(\dagger)$. Then the perverse filtration associated with $f$ admits a strongly multiplicative splitting.
\end{prop}

\begin{proof}
It suffices to prove for the three cases in Proposition \ref{str}. By Proposition \ref{finite}, it suffices to prove for $A\to E$, $E\times \BC\to \BC$ and $(E\times \BC^*)/\Gamma\to \BC^*/\Gamma$.

Since $A\to E$ is a smooth morphism, we have
\[
Rf_*\BQ_A=\BQ_E\bigoplus \BQ^2_E[-1]\bigoplus \BQ_E[-2],
\]
and hence the perverse numbers are
\begin{center}
    \begin{tabular}{c|c|c|c}
    $\dim$ & $P_0$ & $\textup{Gr}_1^P$ & $\textup{Gr}_2^P$\\
    \hline
    $H^0$  & 1 &0 & 0\\
    \hline
    $H^1$  & 2 & 2 & 0\\
    \hline
    $H^2$  & 1 & 4&  1\\
    \hline
    $H^3$  & 0 & 2 & 2\\
    \hline
    $H^4$  & 0 &0 & 1\\
    \end{tabular}
\end{center}
Let $\alpha,\beta\in P_0H^1(A,\BQ)$, $\gamma,\delta\in H^1(A,\BQ)$ be a basis of $H^1(A,\BQ)$ adapted to the perverse filtration. Since the perverse filtration of $f$ is multiplicative \cite[Proposition 4.17]{Z}, we have 
\begin{equation}\label{ijkl}
\Fp(\alpha\beta)=0, ~~~~\Fp(\alpha\gamma),\Fp(\alpha\delta),\Fp(\beta\gamma),\Fp(\beta\delta)\le1,~~~~ \Fp(\gamma\delta)\le 2.
\end{equation}
By the ring structure of $H^*(A,\BQ)$, they form a basis of $H^2(A,\BQ)$. Comparing with the perverse numbers, we see that the equalities of (\ref{ijkl}) must hold. They give a splitting of the perverse filtration $P_\bullet H^2(A,\BQ)$.
\[
H^2(A,\BQ)=\langle\alpha\beta\rangle\bigoplus\langle\alpha\gamma,\alpha\delta,\beta\gamma,\beta\delta\rangle\bigoplus \langle\gamma\delta\rangle.
\]
The same argument works for $H^3(A,\BQ)$ and it is straightforward to see that the induced splitting on the exterior algebra generated by $\alpha,\beta,\gamma,\delta$ is strongly multiplicative.

The perverse filtration of $E\times \BC\to \BC$ coincides with the cohomological degree, which is obviously strongly multiplicative.

Since the $\Gamma$-action on $H^*(E\times\BC^*,\BQ)$ is trivial, the perverse filtration associated with $(E\times\BC^*)/\Gamma\to \BC^*/\Gamma$ is identical to the one of $E\times \BC^*\to\BC^*$, and hence is strongly multiplicative.
\end{proof}

\subsection{Generalized Kummer varieties}
Let $X$ be a connected quasi-projective commutative group scheme. The summation map $+:X^n\to X$ descends to the symmetric product $+:X^{(n)}\to X$ and we define
\begin{equation} \label{1000}
    X^{((n))}=\ker\left(X^{(n)}\xrightarrow{+}X\right).
\end{equation}

For a given partition $\nu=1^{a_1}\cdots n^{a_n}$ of $n$, we define $X^{((\nu))}$ as the base change

\[
\begin{tikzcd}
X^{((\nu))}\arrow[d]\arrow[r,"\iota^{((\nu))}"] &X^{((n))}\arrow[d,hook]\\
X^{(\nu)}\arrow[r,"\iota^{(\nu)}"] & X^{(n)}.
\end{tikzcd}
\]

Let  
\begin{equation}\label{1002}
    X^{\nu}_0=\left\{(x_{ij})_{1\le i\le n,1\le j\le a_i}\mid x_{ij}\in X,\,\sum_{i=1}^n\sum_{j=1}^{a_i}ix_{ij}=0\right\}\subset X^{\nu},
\end{equation}
then $X^\nu_0$ admits a natural $\mathfrak{S}_\nu$-action and the quotient is $X^{((\nu))}$. Denote 
\[
\vec{x}=(x_{ij})_{1\le i\le n, 1\le j\le a_i}\in X^{\nu},
\]
and 
\[
\vec{v}=(v_{ij})_{1\le i\le n,1\le j\le a_i}\in\BZ^{\nu},
\]
where $v_{ij}=i$. Then by (\ref{1002}), we have
\[
X^{\nu}_0=\{\vec{x}\in X^{\nu}\mid\vec{v}\cdot\vec{x}=0\}.
\]
Let 
\[
\vec{v}'=\frac{1}{\gcd(\nu)}\vec{v},
\]
and $X[\gcd(\nu)]$ be the $\gcd(\nu)$-torsion points of $X$. Then $\vec{v}\cdot\vec{x}=0$ is equivalent to $\vec{v}'\cdot\vec{x}\in X[\gcd(\nu)]$, and hence
\begin{equation}\label{1234}
    X_0^{\nu}=\bigsqcup_{\sigma\in X[\gcd(\nu)]}\{\vec{x}\in X^{\nu}\mid \vec{v}'\cdot\vec{x}=\sigma\}.
\end{equation}
For later use, we denote $\tilde{X}^\nu_\sigma=\{\vec{x}\in A^{\nu}\mid \vec{v}'\cdot\vec{x}=\sigma\}$ for short.

When $\dim X=2$, we define the generalized Kummer variety $X^{[[n]]}$ by the Cartesian diagram 
\[
\begin{tikzcd}
X^{[[n]]}\arrow[r]\arrow[d,"\pi'"]& X^{[n]}\arrow[d,"\pi"]\\
X^{((n))}\arrow[r]& X^{(n)}.
\end{tikzcd}
\]

Let $f:A\to C$ be a morphism satisfying ($\dagger$). Then by Proposition \ref{str}, $f$ factors as
\[
A\xrightarrow{f'}B\xrightarrow{g}C
\]
where $B$ is a smooth group scheme of dimension 1, $f'$ is a homomorphism of algebraic groups and $g$ is a finite morphism. Then there is a natural map
\[
f'^{((n))}:A^{((n))}\to B^{((n))}.
\]
By abuse of notation, denote by $C^{((n))}$ the image of the composition
\[
A^{((n))}\hookrightarrow A^{(n)}\xrightarrow{f^{(n)}} C^{(n)},
\]
and denote
\[
f^{((n))}:A^{((n))}\xrightarrow{f'^{((n))}} B^{((n))}\xrightarrow{g^{((n))}}C^{((n))}.
\]

\begin{prop}\label{dim}
Let $f:A\to C$ be a morphism satisfying ($\dagger$). Then $g^{((n))}:B^{((n))}\to C^{((n))}$ is a finite surjective morphism. In particular, $\dim C^{((n))}=n-1$ and the perverse filtration associated with $h':A^{[[n]]}\to C^{((n))}$ is identical to the one associated with $h'_B:A^{[[n]]}\to B^{((n))}$.
\end{prop}

\begin{proof}
Since $B\to C$ is finite, the composition 
\[
B^{((n))}\hookrightarrow B^{(n)}\xrightarrow{g^{(n)}} C^{(n)}
\]
is finite. By definition, $g^{((n))}$ is the morphism onto its image and hence is finite surjective. By (\ref{1000}), $\dim B^{((n))}=n-1$. So $\dim C^{((n))}=n-1$. The identification of perverse filtration follows from Proposition \ref{finite}.
\end{proof}

Using the notations introduced above, we have the commutative diagram

\[
\begin{tikzcd}
& &A^{[n]}\arrow[d,"\pi"]& \\
A^\nu\arrow[r,"/\mathfrak{S}_\nu"]\arrow[d,swap,"f'^\nu"]&A^{(\nu)}\arrow[r,"\iota^{(\nu)}"]&A^{(n)}&A^{[[n]]}\arrow[d,swap,"\pi'"]\arrow[dd,bend left,"h'_B"]\arrow[lu]\\
B^\nu&A^{\nu}_0\arrow[r,"/\mathfrak{S}_{\nu}"]\arrow[d,"f'^\nu_0"]\arrow[lu]&A^{((\nu))}\arrow[d,"f'^{((\nu))}"]\arrow[r,"\iota^{((\nu))}_A"]\arrow[lu]&A^{((n))}\arrow[d,swap,"f'^{((n))}"]\arrow[lu]\\
&B^{\nu}_0\arrow[r,"/\mathfrak{S}_{\nu}"]\arrow[lu]& B^{((\nu))}\arrow[r,"\iota^{((\nu))}_B"]&B^{((n))}.
\end{tikzcd}
\]

By the proper base change theorem and Theorem \ref{hilb}, we have
\begin{equation}\label{1003}
\begin{split}
Rh'_{B*}\BQ_{A^{[[n]]}}[2n]=&Rf'^{((n))}_*\bigoplus_\nu\iota_{A*}^{((\nu))}\BQ_{A^{((\nu))}}[2l(\nu)]\\
=&\bigoplus_\nu \iota_{B*}^{((\nu))}Rf'^{((\nu))}_*\BQ_{A^{((\nu))}}[2l(\nu)]\\
=&\bigoplus_\nu \iota_{B*}^{((\nu))}(Rf'^{\nu}_{0*}\BQ_{A^\nu_0})^{\mathfrak{S}_\nu}[2l(\nu)]
\end{split}
\end{equation}

Since all $\iota$'s involved are $t$-exact, to study the perverse filtration associated with $h'_B$, it suffices to study the perverse filtration associated with $f'^{\nu}_0:A^\nu_0\to B^\nu_0$ together with its $\mathfrak{S}_\nu$-action for all partition $\nu$ of $n$.

\subsection{Topology of the morphism $f'^\nu_0$}
In this section, we will calculate the perverse filtration associated with $f'^\nu_0$ and describe the $\mathfrak{S}_\nu$-action on it. We first calculate the Betti numbers of $A^\nu_0$.

\begin{prop} \label{non-can}
Let $\nu$ be a partition of $n$. There exists a non-canonical isomorphism 
\[
A^\nu_0\cong A[\gcd(\nu)]\times A^{l(\nu)-1},
\]
In particular, there is an isomorphism of cohomology groups
\[
H^*(A^\nu_0,\BQ)\cong \bigoplus_{\sigma\in A[\gcd(\nu)]} H^*(A,\BQ)^{\otimes l(\nu)-1}. 
\]
and 
\begin{equation}\label{1253}
\dim H^*(A^\nu_0,\BQ)=|A[\gcd(\nu)]|\cdot (\dim H^*(A,\BQ))^{l(\nu)-1}.
\end{equation}

\end{prop}

\begin{proof}
We follow the notation in Section 4.2. Since $\vec{v}'$ is primitive in the lattice $\BZ^{\nu}$, we may extend $\vec{v}'$ to be a basis of $\BZ^{\nu}$. Equivalently, there is an invertible $l(\nu)\times l(\nu)$ matrix $M$ with integer entries whose first row is $\vec{v}'$. We define $M:A^{\nu}\to A^{\nu}$ formally by the rule of linear transformations. Then $\vec{v}'\cdot \vec{x}=\sigma$ for a $\gcd(\nu)$-torsion point $\sigma$ if and only if the first entry of $M(\vec{x})$ is $\sigma$. Therefore, $M$ maps $A^{\nu}_0$ isomorphically to $A[\gcd(\nu)]\times A^{l(\nu)-1}$, \emph{i.e.}
\[
\begin{tikzcd}
A^{\nu}_0\arrow[d,hook]\arrow[r,"M"]&A[\gcd(\nu)]\times A^{l(\nu)-1}\arrow[d,hook]\\
A^{\nu}\arrow[r,"M"]&A^{\nu}.
\end{tikzcd}
\]
By K\"unneth formula,
\[
H^*(A^\nu_0,\BQ)\cong \bigoplus_{\sigma\in A[\gcd(\nu)]} H^*(A,\BQ)^{\otimes l(\nu)-1}.
\] 
\end{proof}

The $\mathfrak{S}_\nu$-action on the right side is obscure, so it is difficult to descent directly to the perverse filtration of $f'^{((\nu))}:A^{((\nu))}\to B^{((\nu))}$. Therefore, we switch to analysis the product of morphisms
\[
f'^\nu_0\times f:A^\nu_0\times A\to B^\nu_0\times B,
\]
which behaves better with respect to the $\mathfrak{S}_\nu$-action. Consider the diagram
\[
\begin{tikzcd}
A^\nu_0\times A\arrow[r,"\Sigma"]\arrow[d,"f'^\nu_0\times f"]& A^{\nu}\arrow[d,"f'^{\nu}"]\\
B^\nu_0\times B\arrow[r,"\Sigma"]&B^{\nu},
\end{tikzcd}
\]
where $\Sigma(\vec{x},a)=(x_{ij}+a)_{1\le i\le n,1\le j\le a_i}$. Then $\Sigma$ is an $\mathfrak{S}_\nu$-equivariant \'etale covering. 

\begin{prop} \label{aaa}
Let $f':A\to B$ be the morphism obtained in Proposition \ref{str}. Then there is a canonical $\mathfrak{S}_\nu$-equivariant isomorphism
\[
H^*(A_0^\nu\times A,\BQ)=\bigoplus_{\sigma\in A[\gcd(\nu)]} H^*(A^{\nu},\BQ)
\]
where $\mathfrak{S}_\nu$ acts on the factor $H^*(A_0^\nu,\BQ)$ on the left side. Under this isomorphism, we have
\[
    P_k^{f'^{\nu}_0\times f'}H^*(A_0^\nu\times A,\BQ)=\bigoplus_{\sigma\in A[\gcd(\nu)]} P^{f'^{\nu}}_kH^*(A^{\nu},\BQ).
\]
In particular, we have 
\[
    P_k^{f'^{((\nu))}\times f'}H^*(A^{((\nu))}\times A,\BQ)=\bigoplus_{\sigma\in A[\gcd(\nu)]} P^{f'^{(\nu)}}_kH^*(A^{(\nu)},\BQ).
\]
\end{prop}

\begin{proof}
By (\ref{1234}), the map $\Sigma$ is a disjoint union of \'etale coverings $\Sigma_\sigma:\tilde{A}_\sigma^\nu\times A\to A^{\nu}$ for all $\sigma\in A[\gcd(\nu)]$. We have
\begin{equation} \label{1234.5}
    \Sigma_*\BQ_{A^\nu_0\times A}=\bigoplus_{\sigma\in A[\gcd(\nu)]} \Sigma_{\sigma*} \BQ_{\tilde{A}_\sigma^\nu\times A}=\bigoplus_{\sigma\in A[\gcd(\nu)]}(\BQ_{A^{\nu}}\oplus F_\sigma),
\end{equation}
for some sheaves $F_\sigma$. By taking the hyper-cohomology, we have
\begin{equation}\label{1235}
    H^*(A_0^\nu\times A,\BQ)=\bigoplus_{\sigma\in A[\gcd(\nu)]} H^*(A^{\nu},\BQ)\oplus\BH(F_\sigma).
\end{equation}

The summands $H^*(A^\nu,\BQ)$ is canonical since $\BQ_{A^\nu}$ is canonical in the decomposition (\ref{1234.5}). Proposition \ref{non-can} and K\"unneth formula yields an identification of Betti numbers
\begin{equation} \label{1236}
   \dim H^*(A^\nu_0\times A,\BQ)= |A[\gcd(\nu)]|\cdot \dim H^*(A,\BQ)^{\otimes l(\nu)}. 
\end{equation}
Comparing \ref{1235} and (\ref{1236}), we see that $\BH(F_\sigma)=0$ for all $\sigma$. Now we use the definition (\ref{01}) to calculate the perverse filtration associated with $f'^{\nu}_0$. 
Since $\Sigma$ is finite, and hence is $t$-exact, applying $^{\Fp}\tau_{\le k}Rf'^{\nu}_*$ to (\ref{1234.5}) yields
\[
\Sigma_* {^\mathfrak{\Fp}}\tau_{\le k}R(f'^\nu_0\times f)_*\BQ_{A^\nu_0\times A}=\bigoplus_{\sigma\in A[\gcd(\nu)]}\left({^\mathfrak{\Fp}}\tau_{\le k}Rf'^{\nu}_*\BQ_{A^{\nu}}\oplus {^\mathfrak{\Fp}}\tau_{\le k}Rf'^{\nu}_*F_\sigma\right).
\]
The hypercohomology yields
\[
\begin{tikzcd}
\BH({^\mathfrak{\Fp}}\tau_{\le k}R(f'^\nu_0\times f)_*\BQ_{A^\nu_0\times A})\arrow[r,equal]\arrow[d]&
\displaystyle\bigoplus_{\sigma\in A[\gcd(\nu)]}\BH\left({^\mathfrak{\Fp}}\tau_{\le k}Rf'^{\nu}_*\BQ_{A^{\nu}}\oplus {^\mathfrak{\Fp}}\tau_{\le k}Rf'^{\nu}_*F_\sigma\right)\arrow[d]\\
\BH(R(f'^\nu_0\times f)_*\BQ_{A^\nu_0\times A})\arrow[r,equal]\arrow[d,equal]&
\displaystyle\bigoplus_{\sigma\in A[\gcd(\nu)]}\BH\left(Rf'^{\nu}_*\BQ_{A^{\nu}}\oplus Rf'^{\nu}_*F_\sigma\right)\arrow[d,equal]\\
H^*(A^\nu_0\times A,\BQ)\arrow[r,equal]& \displaystyle\bigoplus_{\sigma\in A[\gcd(\nu)]}H^*(A^\nu,\BQ).
\end{tikzcd}
\]
By the definition of perverse filtrations (\ref{01}), we have an isomorphism
\begin{equation}\label{1237}
    P_k^{f'^\nu_0\times f'}H^*(A_0^\nu\times A,\BQ)=\bigoplus_{\sigma\in A[\gcd(\nu)]} P_k^{f'^\nu}H^*(A^{\nu},\BQ).
\end{equation}
The isomorphism (\ref{1237}) is $\mathfrak{S}_\nu$-equivariant since $\Sigma_\sigma$ is for all $\sigma\in A[\gcd(\nu)]$, and is canonical since the summmands
\[
H^*(A^{\nu},\BQ)\to H^*(A_0^\nu\times A,\BQ)
\]
are canonical for all $\sigma$ by the decomposition theorem.
\end{proof}

\subsection{Perverse filtration and multiplicativity}

\begin{thm} \label{thm}
Let $f:A\to C$ be a morphism satisfying $(\dagger)$. Let $h':A^{[[n]]}\to C^{((n))}$ be the induced morphism. Then there is a canonical decomposition
\[
H^*(A^{[[n]]}\times A,\BQ)=\bigoplus_\nu\bigoplus_{\sigma \in A[\gcd(\nu)]} H^*(A^{(\nu)},\BQ)[2l(\nu)-2n].
\]
Under this identification, we have
\[
P_k^{h'\times f} H^*(A^{[[n]]}\times A,\BQ)=\bigoplus_\nu\bigoplus_{\sigma \in A[\gcd(\nu)]} P_{k+l(\nu)-n}^{f^{(\nu)}} H^*(A^{(\nu)},\BQ)[2l(\nu)-2n].
\]
\end{thm}

\begin{proof}
By (\ref{1003}) and K\"unneth formula,
\begin{equation}\label{1238}
\begin{split}
R(h'_{B}\times f')&_*\BQ_{A^{[[n]]}\times A}[2n]\\
=&\bigoplus_\nu \left(\iota_{B}^{((\nu))}\times \textup{id}\right)_*\left(R(f'^{\nu}_{0}\times f')_*\BQ_{A^\nu_0\times A}\right)^{\mathfrak{S}_\nu}[2l(\nu)]
\end{split}
\end{equation}

By Proposition \ref{aaa}, we have
\[
\begin{split}
H^*({A^{[[n]]}\times A},\BQ)[2n]=&\bigoplus_\nu H^*(A^\nu_0\times A,\BQ)^{\mathfrak{S}_\nu}[2l(\nu)].\\
=&\bigoplus_{\nu,\sigma} H^*(A^\nu,\BQ)^{\mathfrak{S}_\nu}[2l(\nu)]\\
=&\bigoplus_{\nu,\sigma} H^*(A^{(\nu)},\BQ)[2l(\nu)].
\end{split}
\]
To calculate the perverse filtration associated with $h'_B\times f'$, we apply the morphism of functors $\BH\circ{^\Fp\tau_{\le k}}\to \BH$ to (\ref{1238}). Since $h'_B\times f'$ is a flat morphism of relative dimension $n$, the left side calculates
\[
P_{k+n}^{h'_B\times f'}H^*(A^{[[n]]}\times A,\BQ)[2n].
\]
Since $f'^{\nu}_{0}\times f'$ is a flat morphism of relative dimension $l(\nu)$, the right side calculates
\[
\begin{split}
    &\bigoplus_{\nu}P_{k+l(\nu)}^{f'^\nu_0\times f'} H^*(A^\nu_0\times A,\BQ)^{\mathfrak{S}_\nu}[2l(\nu)]\\
    =&\bigoplus_{\nu}P_{k+l(\nu)}^{f'^{((\nu))}\times f'} H^*(A^{((\nu))}\times A,\BQ)[2l(\nu)]\\
    =&\bigoplus_{\nu,\sigma}P_{k+l(\nu)}^{f'^{(\nu)}} H^*(A^{(\nu)},\BQ)[2l(\nu)], \textup{ (by Proposition \ref{aaa}).}
\end{split}
\]
Therefore, we have the identification of perverse filtrations
\[
P_{k}^{h'_B\times f'}H^*(A^{[[n]]},\BQ)=\bigoplus_{\nu,\sigma}P_{k+l(\nu)-n}^{f'^{(\nu)}} H^*(A^{(\nu)},\BQ)[2l(\nu)-2n].
\]
By Proposition \ref{Kunneth} and Proposition \ref{dim}, the perverse filtration associated with $h'_B\times f'$ is identical to the one associated with $h'\times f$. By Proposition \ref{finite}, the finiteness of $g^{(\nu)}:B^{(\nu)}\to C^{(\nu)}$ implies that the perverse filtration associated with  $f'^{(\nu)}$ is identical with $f^{(\nu)}$. We have
\[
P_{k}^{h'\times f}H^*(A^{[[n]]}\times A,\BQ)=\bigoplus_{\nu,\sigma}P_{k+l(\nu)-n}^{f^{(\nu)}} H^*(A^{(\nu)},\BQ)[2l(\nu)-2n].
\]
\end{proof}

To state the cup product of generalized Kummer varieties, we introduce the following notations.
\begin{defn} \label{dfn}
Let $\nu$ be a partition of $n$, and $\alpha_\nu\in H^*(A^{(\nu)},\BQ)$ be a cohomology class. Let $\sigma\in A[n]$ be an $n$-torsion of $A$. Denote by $\alpha_{\nu,\sigma}$ the image of $\alpha_\nu$ in $H^*(A^{[[n]]}\times A,\BQ)$ via the decomposition
\[
H^*(A^{[[n]]}\times A,\BQ)=\bigoplus_{\nu,\sigma}H^*(A^{(\nu)},\BQ)[2l(\nu)-2n]
\]
if $\sigma\in A[\gcd(\nu)]$ and 0 otherwise.
\end{defn}

When $\nu$ runs through all partitions of $n$, $\alpha_\nu$ runs through a basis of $H^*(A^{(\nu)},\BQ)$ for each $\nu$, and $\sigma$ runs through $A[\gcd(\nu)]$, the classes $\alpha_{\nu,\sigma}$ form a basis of the cohomology group $H^*(A^{[[n]]}\times A,\BQ)$. It suffices to describe the cup product on this basis. By \cite[Proposition 4.1]{FG} and \cite[Corollary 6.15]{FTV} and \cite[Theorem 1.7]{N}, the cup product formula for generalized Kummer varieties is described as follows.

\begin{thm}   \label{cup}
Let $A$ be a smooth connected quasi-projective commutative group scheme of dimension 2. Let  $\alpha_{\nu}\in H^*(A^{(\nu)},\BQ)$ and $\beta_{\mu}\in H^*(A^{(\mu)},\BQ)$ be two cohomology classes. Following the notations in Corollary \ref{decomp}, suppose 
\begin{equation}\label{1240}
\alpha_\nu^{[n]}\cdot\beta_\mu^{[n]}=\sum_\lambda \gamma_\lambda^{[n]}
\end{equation}
is the cup product in $H^*(A^{[n]},\BQ)$. Then the product in $H^*(A^{[[n]]}\times A,\BQ)$ is calculated as
\[
\alpha_{\nu,\sigma}\cdot\beta_{\mu,\tau}=\sum_{\lambda}\gamma_{\lambda,\sigma\tau}.
\]
\end{thm}

\begin{thm} \label{main}
Let $f:A\to C$ be a proper surjective morphism from a connected quasi-projective
commutative group scheme $A$ of dimension 2 to a quasi-projective curve $C$. Then the perverse filtration associated with the induced morphism $h':A^{[[n]]}\to C^{((n))}$ is multiplicative.
\end{thm}

\begin{proof}
Let $\alpha\in H^*(A^{(\nu)},\BQ)$. Then by Proposition \ref{4.7} and Theorem \ref{thm}, we have 
\begin{equation}\label{1241}
\Fp^{h'\times f}(\alpha_{\nu,\sigma})=\Fp^{f^{(\nu)}}(\alpha_\nu)+n-l(\nu)=\Fp^h(\alpha_{\nu}^{[n]})
\end{equation}
for any $\sigma\in A[\gcd(\nu)]$. By Theorem \ref{multi}, the perverse filtration associated with $\pi$ is multiplicative 
\[
\Fp^h(\alpha_{\nu}^{[n]}\cdot\beta_{\mu}^{[n]})\le \Fp^h(\alpha_{\nu}^{[n]}) +\Fp^h(\beta_{\mu}^{[n]}).
\]
By (\ref{1240}) and Corollary \ref{decomp}, 
\[
\Fp^h(\gamma_{\lambda}^{[n]})\le \Fp^h(\alpha_{\nu}^{[n]}) +\Fp^h(\beta_{\mu}^{[n]})
\]
holds for each class $\gamma_{\lambda}^{[n]}$ occurring in (\ref{1240}). By (\ref{1241}), 
\[
\Fp^{h'\times f}(\gamma_{\lambda,\sigma\tau})\le \Fp^{h'\times f}(\alpha_{\nu,\sigma})+ \Fp^{h'\times f}(\beta_{\mu,\tau})
\]
holds for each non-zero $\gamma_{\lambda,\sigma\tau}$. So
\begin{equation} \label{1242}
    \Fp^{h'\times f}(\alpha_{\nu,\sigma}\cdot \beta_{\mu,\tau})\le \Fp^{h'\times f}(\alpha_{\nu,\sigma})+ \Fp^{h'\times f}(\beta_{\mu,\tau}).
\end{equation}
Therefore, the perverse filtration associated with $h'\times f$ is multiplicative.  The flatness of $f:A\to C$ implies $\Fp^f(1)=0$. Then by Proposition \ref{Kunneth} and (\ref{1242}),
\[
\begin{split}
   \Fp^{h'}(\alpha\beta)=\Fp^{h'\times f}(\alpha\beta\boxtimes 1)=\Fp^{h'\times f}(\alpha\boxtimes 1 \cdot \beta\boxtimes 1)\\
   \le \Fp^{h'\times f}(\alpha\boxtimes 1)+\Fp^{h'\times f}(\beta\boxtimes 1)=\Fp^{h'}(\alpha)+\Fp^{h'}(\beta).
\end{split}
\]
\end{proof}

\begin{cor}
Let $f:A\to C$ be a morphism satisfying the condition of Theorem \ref{main}. Then the perverse filtration associated with the induced morphism $h':A^{[[n]]}\to C^{((n))}$ admits a natural strongly multiplicative splitting in the sense of (\ref{001}).
\end{cor}

\begin{proof}
By Proposition \ref{4.3}, the perverse filtration associated with $f:A\to C$ admits a strongly multiplicative splitting. Such a splitting induces a splitting of the perverse filtration associated with $f^{(\nu)}:A^{(\nu)}\to C^{(\nu)}$ for any partition $\nu$, and a splitting of the perverse filtration associated with $h:A^{[n]}\to C^{(n)}$; see \cite[Section 2.2]{SZ}. Now Proposition \ref{aaa} and Theorem \ref{thm} produce splittings of the perverse filtrations associated with $f'^{(\nu)}\times f':A^{(\nu)}_0\times A\to B^{(\nu)}_0\times B$ and $h'\times f:A^{[[n]]}\times A\to C^{((n))}\times C$, respectively. It follows directly from the construction that if $\alpha_\nu\in H^*(A^{(n)},\BQ)$ is a pure class with respect to the splitting, then $\alpha_\nu^{[n]}$ and $\alpha_{\nu,\sigma}$ are both pure in the corresponding splittings, and the equation (\ref{1241}) still holds. By \cite[Proposition 2.8]{SZ}, the perverse filtration associated with $h$ is strongly multiplicative, the inequalities in the proof of Theorem \ref{main} are equations for pure classes $\alpha_\nu^{[n]}$, $\beta_\nu^{[n]}$, $\alpha_{\nu,\sigma}$, and $\beta_{\nu,\tau}$. Therefore, the splitting we constructed is strongly multiplicative. 
\end{proof}

\end{document}